\documentclass[a4paper,11pt]{article}
\usepackage[utf8x]{inputenc}

\usepackage{xfrac}
\usepackage{parskip}
\usepackage{a4wide}
\usepackage{amssymb}
\usepackage{amsmath}
\usepackage{amsthm}
\usepackage[mathscr]{euscript}
\usepackage[normalem]{ulem}
\usepackage[colorlinks=true]{hyperref}
\usepackage[usenames,dvipsnames]{xcolor}
\usepackage{appendix}
\usepackage{tikz}
\usetikzlibrary{arrows,angles}
\usetikzlibrary{matrix}
\usetikzlibrary{decorations}
\usetikzlibrary{arrows,calc,shapes,decorations.pathreplacing}
\usepackage{tikz-cd}
\usepackage{todonotes}
\hypersetup{hidelinks}

\numberwithin{equation}{section}

\theoremstyle{definition}

\newtheorem{Definition}{Definition}[section]

\newtheorem{Remark}[Definition]{Remark}

\newtheorem{Conjecture}[Definition]{Conjecture}

\theoremstyle{plain}

\newtheorem{Theorem}[Definition]{Theorem}
\newtheorem{Proposition}[Definition]{Proposition}
\newtheorem{Lemma}[Definition]{Lemma}

\newcommand{\R}{\mathbb R}
\newcommand{\N}{\mathbb N}

\newcommand{\cR}{\mathcal{R}}

\newcommand{\eps}{\varepsilon}

\usepackage[xcolor]{changebar}
\cbcolor{blue}

\newcommand{\comp}{\Subset}

\newcommand{\D}{\mathcal{D}}

\newcommand{\Riem}{\mathrm{Riem}}
\DeclareMathOperator{\CBB}{CBB}

\DeclareMathOperator{\CAT}{CAT}
\DeclareMathOperator{\inj}{inj}
\DeclareMathOperator{\CR}{CR} 
\DeclareMathOperator{\SR}{SR} 
\DeclareMathOperator{\tnorm}{tnorm} 
\DeclareMathOperator{\Img}{Im}
\DeclareMathOperator{\Hess}{Hess}
\newcommand{\loc}{\mathrm{loc}}
\newcommand*\dif{\mathop{}\!\mathrm{d}}

\newcommand*{\cang}[1]{\Tilde\measuredangle^{#1}}

\usepackage[inline]{enumitem}
\newcommand{\enumlabelformat}{\roman}
\newcommand{\enumlabelfont}[1]{#1}
\newlength{\thelabelsep}
\setlist{labelsep=\thelabelsep}
\setlist[enumerate,1]{font=\enumlabelfont,label=(\enumlabelformat*),leftmargin=2.5em}
\setlist[itemize]{leftmargin=2.5em,label=$-$}

\newcounter{inlineenum}
\renewcommand{\theinlineenum}{\enumlabelformat{inlineenum}}

\let\epsilon\varepsilon
\let\phi\varphi

\makeatletter
\let\save@mathaccent\mathaccent
\newcommand*\if@single[3]{%
  \setbox0\hbox{${\mathaccent"0362{#1}}^H$}%
  \setbox2\hbox{${\mathaccent"0362{\kern0pt#1}}^H$}%
  \ifdim\ht0=\ht2 #3\else #2\fi
  }
\newcommand*\rel@kern[1]{\kern#1\dimexpr\macc@kerna}
\newcommand*\widebar[1]{\@ifnextchar^{{\wide@bar{#1}{0}}}{\wide@bar{#1}{1}}}
\newcommand*\wide@bar[2]{\if@single{#1}{\wide@bar@{#1}{#2}{1}}{\wide@bar@{#1}{#2}{2}}}
\newcommand*\wide@bar@[3]{%
  \begingroup
  \def\mathaccent##1##2{%
    \let\mathaccent\save@mathaccent
    \if#32 \let\macc@nucleus\first@char \fi
    \setbox\z@\hbox{$\macc@style{\macc@nucleus}_{}$}%
    \setbox\tw@\hbox{$\macc@style{\macc@nucleus}{}_{}$}%
    \dimen@\wd\tw@
    \advance\dimen@-\wd\z@
    \divide\dimen@ 3
    \@tempdima\wd\tw@
    \advance\@tempdima-\scriptspace
    \divide\@tempdima 10
    \advance\dimen@-\@tempdima
    \ifdim\dimen@>\z@ \dimen@0pt\fi
    \rel@kern{0.6}\kern-\dimen@
    \if#31
      \overline{\rel@kern{-0.6}\kern\dimen@\macc@nucleus\rel@kern{0.4}\kern\dimen@}%
      \advance\dimen@0.4\dimexpr\macc@kerna
      \let\final@kern#2%
      \ifdim\dimen@<\z@ \let\final@kern1\fi
      \if\final@kern1 \kern-\dimen@\fi
    \else
      \overline{\rel@kern{-0.6}\kern\dimen@#1}%
    \fi
  }%
  \macc@depth\@ne
  \let\math@bgroup\@empty \let\math@egroup\macc@set@skewchar
  \mathsurround\z@ \frozen@everymath{\mathgroup\macc@group\relax}%
  \macc@set@skewchar\relax
  \let\mathaccentV\macc@nested@a
  \if#31
    \macc@nested@a\relax111{#1}%
  \else
    \def\gobble@till@marker##1\endmarker{}%
    \futurelet\first@char\gobble@till@marker#1\endmarker
    \ifcat\noexpand\first@char A\else
      \def\first@char{}%
    \fi
    \macc@nested@a\relax111{\first@char}%
  \fi
  \endgroup
}
\makeatother

\newcommand\restr[2]{{
  \left.\kern-\nulldelimiterspace 
  #1 
  \vphantom{\big|} 
  \right|_{#2} 
  }}

\title{Distributional sectional curvature bounds for Riemannian metrics of low regularity}

\author{Darius Erös\thanks{{\tt darius.eroes@univie.ac.at}, Faculty of Mathematics, University of Vienna, Austria.}\\Michael Kunzinger\thanks{{\tt michael.kunzinger@univie.ac.at}, Faculty of Mathematics, University of Vienna, Austria.}\\Argam Ohanyan\thanks{{\tt argam.ohanyan@utoronto.ca}, Department of Mathematics, University of Toronto, Ontario, Canada.}\\Alessio Vardabasso\thanks{{\tt alessio.vardabasso@univie.ac.at}, Faculty of Mathematics, University of Vienna, Austria.}
}
\begin{document}

\date{\today}


\maketitle

\begin{abstract}
Sectional curvature bounds are of central importance in the study of Riemannian mani\-folds, both in smooth differential geometry and in the generalized synthetic setting of Alexandrov spaces. Riemannian metrics along with metric spaces of bounded sectional curvature enjoy a variety of, oftentimes rigid, geometric properties. The purpose of this article is to introduce and discuss a new notion of sectional curvature bounds for manifolds equipped with continuous Riemannian metrics of Geroch--Traschen regularity, i.e., $H^1_{\mathrm{loc}} \cap C^0$, based on a distributional version of the classical formula. Our main result states that for $g \in C^1$, this new notion recovers the corresponding bound based on triangle comparison in the sense of Alexandrov. A weaker version of this statement is also proven for locally Lipschitz continuous metrics.

\vspace{1em}

\medskip


\end{abstract}
\tableofcontents

\section{Introduction}\label{section: intro}

Sectional curvature bounds are of central importance in the study of Riemannian manifolds. Spaces with bounds on the sectional curvature either from above or below enjoy a variety of geometric and topological properties which are usually quite rigid, i.e., slight variations in the curvature bounds do not break these properties outside of some special cases.

The central ingredient in the study of sectional curvature bounds is Toponogov's theorem, which establishes an equivalence of bounds on the sectional curvature tensor with comparison inequalities relating certain geometric configurations with the geometry of the constant curvature model spaces. The latter includes comparing distances between points on triangles, angles between geodesics, angle monotonicity, and four-point comparison, to name just a few. As these geometric configurations only involve the metric distance, such a comparison can be viewed as a constraint on general metric (or length) spaces, allowing one to speak of curvature bounds even in this generality.

The study of metric spaces with curvature bounds (also referred to as \textit{Alexandrov spaces}) has led to a rich structure theory and a very useful framework, producing numerous results which are of great importance (and new) even for smooth Riemannian manifolds. A key indicator for the success of this theory is that such notions of sectional curvature bounds for general metric spaces are \textit{stable} under (pointed) Gromov--Hausdorff convergence. Moreover, by Gromov's compactness theorem, the class of Alexandrov spaces with uniform lower curvature bound as well as uniform upper diameter and dimension bounds is compact. We refer to the standard textbooks in the field \cite{burago2001course, BridsonHaefliger, AKPAlexandrov} as well as the references therein for a more detailed discussion on these topics. We will only be interested in continuous Geroch--Traschen metrics, see below.

The objects of study in this article are Riemannian manifolds $(M,g)$ whose metric $g$ is not smooth. In practice, this usually means that the regularity of $g$ is below $C^2$, as for most global comparison-geometric purposes, the metric is not differentiated more than twice. From the point of view of classical analysis, one might be tempted to define curvature tensors for such $g$ by taking distributional derivatives. It turns out that the most ``useful" class of metrics for which this can be done consistently is the \textit{Geroch--Traschen} \cite{GerochTraschen} class (see also \cite{SteinbauerVickersGerochTraschen}), i.e., metrics with regularity $H^1_\mathrm{loc} \cap L^\infty_\mathrm{loc}$ satisfying a certain uniform nondegeneracy condition. Such metrics induce Levi--Civita connections of regularity $L^2_\mathrm{loc}$, from which curvature tensors can be defined consistently in the language of tensor distributions. We refer to the treatments in \cite{LeflochMardare, SteinbauerVickersGerochTraschen, KOSS} for more details.

Since the curvature tensor of a Geroch--Traschen metric is merely a tensor distribution, one cannot use the classical pointwise formula to define sectional curvature bounds analytically. Instead, in Subsection \ref{section:DistributionalBounds}, we introduce such bounds by way of a distributional inequality along smooth vector fields instead of individual tangent vectors. 

Any continuous Geroch--Traschen metric $g$ on a smooth manifold $M$ induces a natural metric length space structure $(M,d_g)$ on $M$, where $d_g$ is obtained, just as for smooth metrics, as the infimum of lengths of curves connecting two points, see \cite{burtscher_length}. Hence, there are two competing notions of sectional curvature bounds available for such metrics, one distributional, introduced analytically in Subsection \ref{section:DistributionalBounds}, and the other synthetic, based on geometric comparison in the sense of classical Alexandrov space theory. It is therefore of considerable interest to investigate the compatibility of these notions.

We aim to initiate the study of such compatibility questions of sectional curvature bounds for low regularity metrics. Our main result and contribution in this direction concerns metrics of regularity $C^1$. For such metrics, we show that distributional sectional curvature bounds from below (resp.\ above) imply the corresponding Alexandrov bounds from below (resp.\ above). Additionally, we present similar but weaker statements for locally Lipschitz continuous metrics. Our primary technique is regularization, which helps us establish approximate (smooth) sectional curvature inequalities for the approximating metrics, for which classical smooth results can be applied.

Moreover, we reconsider the classical examples of P.\ Hartman and A.\ Wintner within our framework to show that, in stark contrast to the case of smooth geometries, distributional sectional curvature bounds for metrics of low regularity may fail to hold in \textit{any} neighborhood of a given point. 
Comparing this phenomenon to the characteristic behavior of geodesics in Alexandrov spaces, these examples will serve as concrete illustrations of the geometric significance of our results.

The article is structured as follows: We begin with a discussion of the basics of low regularity Riemannian metrics (in particular, the continuous Geroch--Traschen class) in Subsection \ref{Subsection: lowregmetric} and give the definition of distributional sectional curvature bounds for such metrics in Subsection \ref{section:DistributionalBounds}. Next, in Subsection \ref{section:Regularization}, we lay out the regularization methods based on manifold-wide convolution which we will use throughout. Furthermore, we establish approximate sectional curvature bounds for the smooth convolved metrics in Proposition \ref{Prop: Localseccurvapprox}. In Subsection \ref{subsection: syntheticalexandrov}, we review the basic notions of synthetic curvature bounds from the theory of Alexandrov spaces, but restrict ourselves only to those formulations of curvature bounds which we intend to use. In the last part of Section \ref{Section: prelim} on preliminary notions and results, i.e., Subsection \ref{subsec: incompleteRiemann}, we generalize some results (including the famous Klingenberg Lemma) to incomplete Riemannian manifolds, which are less often considered in the literature. We will need such generalizations since we often work with local curvature bounds for the smooth metrics on open neighborhoods, which are of course not complete when regarded as Riemannian manifolds. Our main results are presented in Section \ref{Section: distrib implies alexandrov}: In Subsection \ref{subsec: lower} we show that distributional lower sectional curvature bounds by $k \in \R$ for a Riemannian metric $g \in C^1$ imply that $(M,d_g)$ is locally $\CBB(k)$. Similarly, in Subsection \ref{subsec: above}, we establish that a distributional bound from above for such a metric by $k$ implies that $(M,d_g)$ is locally $\CAT(k)$. Moreover, we also prove that metrics of regularity $C^{0,1}_{\loc}$ satisfy variable Alexandrov bounds provided their distributional curvature is bounded. In addition, in Subsection \ref{Subsec: examples}, we discuss the Hartmann--Wintner examples mentioned above within the context of our results. Finally, in Section \ref{section: Conclusion and outlook}, we summarize our results and give an outlook on related open problems and future work.

\section{Preliminary notions and results} \label{Section: prelim}

\subsection{Low regularity Riemannian metrics} \label{Subsection: lowregmetric}

Throughout this paper, we follow the conventions of \cite{petersen} when referring to notions of Riemannian geometry. In particular, the Riemann curvature tensor is defined as 
\[
\Riem(X,Y)Z:=[\nabla_X,\nabla_Y]Z-\nabla_{[X,Y]}Z.
\]
Let us recall some basic properties of Riemannian metrics of low regularity. 
As we intend to define curvature tensors distributionally, we will work within the Geroch--Traschen \cite{GerochTraschen} class, which is the class of metrics of lowest regularity for which this is feasible without losing certain stability properties (see \cite[Sec.\ 4.2]{LeflochMardare}, cf.\ also \cite{SteinbauerVickersGerochTraschen}). In our case, we will only be interested in continuous Geroch--Traschen metrics:

\begin{Definition}[Continuous Geroch--Traschen metrics]
A Riemannian metric $g$ of regularity $H^1_\mathrm{loc}\cap C^0$ on a smooth manifold $M$ is called a \textit{continuous Geroch--Traschen metric}. We will refer to the pair $(M,g)$ as a \textit{(continuous) Geroch--Traschen Riemannian manifold}.
\end{Definition}

A continuous Geroch--Traschen metric gives rise to the Levi--Civita connection $\nabla$ on $TM$ in the usual way (via the Koszul formula), which then itself is of regularity $L^2_\mathrm{loc}$ (i.e.\ all Christoffel symbols $\Gamma^i_{jk}$ are in $L^2_\mathrm{loc}$ for any smooth coordinate chart of $M$). Generally, given any linear connection $\nabla$ on $TM$ of regularity $L^2_\mathrm{loc}$, it is possible to define the corresponding Riemann curvature tensor $\Riem$ by its action on smooth vector fields $X,Y,Z \in C^\infty(TM)$ to produce a distributional vector field $\Riem(X,Y)Z$. In the case where $\nabla$ is the Levi--Civita connection of a Geroch--Traschen metric $g$, we can also give a meaning to the object $g(\Riem(X,Y)Z,W)$ as a distribution on $M$ (\cite[Rem.\ 4.5]{LeflochMardare}, cf.\ also \cite[Cor.\ 5.2]{SteinbauerVickersGerochTraschen}). Moreover, by taking traces with respect to smooth local frames, it is possible to define the corresponding distributional Ricci and scalar curvatures as well. We refer to the discussions in \cite{LeflochMardare} and \cite{KOSS} for more details regarding tensor distributions on manifolds, and for the technical details on how to define curvature tensors starting with $L^2_\mathrm{loc}$-connections.

\begin{Remark}[Higher regularity metrics]
Since our main results in this article concern metrics that are more regular than just Geroch--Traschen, namely $g \in C^1$, let us note that in this case the Levi--Civita connection is $C^0$ and the curvature tensors are tensor distributions of order one (cf.\ the discussion in \cite[Sec.\ 2]{KOSS}). In particular, by inserting $C^1$-vector fields we produce a well-defined scalar distribution (as, locally, multiplying a distribution of order one with a $C^1$-function is well-defined). Moreover, some of our results also concern metrics of local Lipschitz regularity $C^{0,1}_{\mathrm{loc}}$; in this case, the Levi--Civita connection is of class $L^{\infty}_{\mathrm{loc}}$.
\end{Remark}

\begin{Remark}[Low regularity metrics and length spaces]
A continuous Riemannian metric $g$ on a smooth connected manifold $M$ naturally induces the structure of a metric length space $(M,d_g)$. The definition of $d_g$ is the same as in the smooth case via the infimum of lengths of (either piecewise smooth or, equivalently, absolutely continuous) curves connecting two given points. We refer to \cite{burtscher_length}
for proofs.
\end{Remark}

\subsection{Distributional sectional curvature bounds}\label{section:DistributionalBounds}
Having discussed necessary preliminary notions regarding low regularity metrics and connections, we now want to define sectional curvature bounds using a suitable distributional version of the standard formula
\begin{equation*}
    \sec(v,w) = \frac{g(\Riem(v,w)w,v)}{\lvert v \rvert^2 \lvert w \rvert^2 - g(v,w)^2} \qquad \forall \,v,w\in T_p M.
\end{equation*}
Clearly, we cannot just adopt this formula to our setting as the Riemann curvature tensor is no longer defined by its action on single vectors if the regularity of $g$ drops below $C^2$. Instead, we will rely on a definition using smooth vector fields. To conveniently phrase this, given a continuous Geroch--Traschen metric $g$, we rewrite $\Riem$ as a symmetric bilinear form (cf.\ \cite[Sec.\ 3.1.2]{petersen}) by introducing
\begin{align*}
    \cR \colon C^\infty ({\wedge}^2 TM) \times C^\infty ({\wedge}^2 TM) \longrightarrow \D'(M)\\
    \cR(X\wedge Y, V\wedge W):= g (\Riem(X,Y)W,V).
\end{align*}
Here and in what follows, given a function (or distribution) space $\mathcal{F}$ and a vector bundle $E$ over $M$, we denote the space of sections of $E$ of regularity class $\mathcal{F}$ by $\mathcal{F}(E)$. Thus, $C^\infty ({\wedge}^2 TM)$ denotes the space of smooth sections of ${\wedge}^2 TM$.

Let us further define the natural inner product on ${\wedge}^2 TM$ induced by $g$ via
\begin{align*}
    g(X\wedge Y, V\wedge W):= g(X,V)g(Y,W) - g(X,W)g(Y,V).
\end{align*}
With this in place, a sectional curvature bound can conveniently be written as an inequality between the symmetric bilinear forms $\cR$ and the inner product on ${\wedge}^2 TM$:
\begin{Definition}[Distributional sectional curvature bounds]\label{def:distrSec}
Let $(M,g)$ be a continuous Geroch--Traschen Riemannian manifold and let $k \in \R$.  Then $(M,g)$ is said to fulfill a \textit{lower (resp.\ upper) distributional sectional curvature bound} by $k$, symbolically $\sec \geq k$ (resp.\ $\leq$), if
\begin{align}
    \cR \geq k g\quad\text{(resp.}\leq) \quad \text{in the distributional sense}.
\end{align}
Here, ``in the distributional sense'' means that for all $X,Y \in C^\infty(TM)$,
\begin{align}
    \cR(X\wedge Y,X\wedge Y) \geq k g(X\wedge Y, X\wedge Y) \quad \text{(resp.}\leq) \quad\text{in $\D'(M)$}.
\end{align}
More generally, a continuous Geroch--Traschen Riemannian manifold is said to have \emph{distributional sectional curvature locally bounded from below (resp.\ above)} if for all open relatively compact subsets $U\Subset M $ it holds that
\begin{align}
    \cR \geq k_U g\quad\text{(resp.}\leq) \quad \text{in }\D'(U),
\end{align}
with $k_U\in \R$ possibly depending on $U$.
\end{Definition}

Note that if $g \in C^2$, these definitions are equivalent to the usual ones, as is easily seen upon extending given tangent vectors $v,w \in T_pM$ to smooth vector fields $X,Y \in C^\infty(TM)$.

\subsection{Regularization techniques}\label{section:Regularization}

Working with metrics of low regularity, we will repeatedly employ regularizations that result from a general smoothing procedure based on chart-wise convolution with a mollifier (cf.\ \cite[3.2.10]{GKOS01}, \cite[Sec.\ 2]{KSSV}, \cite[Sec.\ 3.3]{graf2020singularity}). To describe it, let $d:=\dim M$ and let
$\rho\in \D(B_1(0))$ (with $B_1(0)$ the unit ball in $\R^d$), $\int \rho = 1$, $\rho\ge 0$ and, for $\eps\in (0,1]$, set $\rho_{\eps}(x):=\eps^{-d}\rho\left (\frac{x}{\eps}\right)$.
Next, fix a countable atlas $\{(U_\alpha, \psi_\alpha) \}_{\alpha\in \N}$ with relatively compact $U_\alpha$, a subordinate smooth partition of unity $\{\xi_\alpha\}_\alpha$ and functions $\chi_\alpha \in C^\infty_c(U_\alpha)$ with $0\leq \chi_\alpha\leq 1$ and $\chi_\alpha \equiv 1$ on a neighborhood of $\mathrm{supp}\, \xi_\alpha $. Then for any tensor distribution $T\in \D'(T^r_s M)$ and $\eps>0$ we define the smooth $(r,s)$-tensor
\begin{equation*}
    T \star_M \rho_\eps := \sum_{\alpha\in\N}\chi_\alpha (\psi_\alpha)^*( ((\psi_\alpha)_*(\xi_\alpha T) )* \rho_\eps).
\end{equation*}
In this expression, the convolution $((\psi_\alpha)_*(\xi_\alpha T) )* \rho_\eps$ is understood component-wise, recalling that each component is in $\D'(\psi_\alpha (U_\alpha))$. 

This ``manifold convolution'' $\star_M$ has smoothing properties that are closely analogous to those of convolution with a mollifier on open subsets of $\R^d$. In particular, for $T\in \D'(T^r_sM)$ we have
$T\star_M \rho_\eps \to T \ \text{ in } \ \D'(T^r_sM)$ as $\eps\to 0$,
and this convergence is even in $C^k_{\mathrm{loc}}$ or $W^{k,p}_{\mathrm{loc}}$ if $T$ is contained in these spaces (see \cite[Prop.\ 3.5]{graf2020singularity}). Moreover, since
$\rho\ge 0$, $\star_M$ preserves non-negativity: $T\in \D'(M), \ T\ge 0  \ \Rightarrow \ T\star_M \rho_\eps \ge 0$ in $C^\infty(M)$. A similar property holds for non-negative tensor distributions of higher valence.

Given a $C^0$-Riemannian metric $g$ on $M$ and fixing $n\in\N_{>0}$,
we now apply the construction detailed in \cite[Rem.\ 2.4]{KOV24} to obtain a smooth map $(\eps,x) \mapsto g^n_\eps(x)$ such that each $g^n_\eps$ is a smooth Riemannian metric on $M$ and such that, for each $K\Subset M$ there exists some $\eps^n_K>0$ with the property that whenever $x\in K$ and $\eps\in (0,\eps^n_K]$ we have $g^n_\eps(x) = g\star_M \rho_\eps(x)$. In addition, we may strengthen (2.4) in \cite[Rem.\ 2.4]{KOV24} to
\begin{equation}\label{eq:distance_sandwich}
\Big(1-\frac{1}{n}\Big) d_g(x,y) \le d_{g^n_\eps}(x,y) \le \Big(1+\frac{1}{n}\Big)d_g(x,y) \qquad \forall x, y \in M \quad \forall \eps\in (0,1]
\end{equation}
for the corresponding Riemannian distances.

Now fix a compact exhaustion $K_j \subset K_{j+1}^\circ$ of $M$ and let $\eps^n_j := \eps^n_{K_j}$ as above. Clearly, we can assume that $\eps^n_j$ is decreasing both in $n$ and in $j$. Finally, let $g_n := g^n_{\eps^n_n}$. 
Then $g_n$ is a smooth Riemannian metric on $M$ and given any $K\Subset M$ there exists some $j\in \N$ such that
$K\subset K_j$. Therefore, $g_n|_K = g\star_M\rho_{\eps^n_n}|_K$ for all $n\ge j$. In particular, $g_n$ has the
same convergence properties as $g_\eps$. Furthermore, due to \eqref{eq:distance_sandwich} it follows that 
\begin{equation}\label{eq:distance_sandwich2}
\Big(1-\frac{1}{n}\Big) d_g(x,y) \le d_{g_n}(x,y) \le \Big(1+\frac{1}{n}\Big)d_g(x,y) \qquad \forall x, y \in M \quad \forall n\in \N.
\end{equation}
In particular, $d_{g_n} \to d_g$
uniformly on compact subsets of $M\times M$ as $n\to \infty$. Henceforth, whenever we talk about approximating metrics $g_n$, we will always mean the ones constructed here. 

The approximating metrics $g_n$ are a very useful tool as they allow us to pass from distributional curvature inequalities for $g$ to approximate smooth inequalities for $g_n$. The main result in this direction we shall use is the following:

\begin{Proposition}[Local sectional curvature bounds for approximating metrics - $C^1$ case] \label{Prop: Localseccurvapprox}
Let $g$ be a $C^1$-Riemannian metric on $M$ with distributional sectional curvature bounded below (resp.\ above) by $k \in \R$. Let $g_n$ be the corresponding approximating metrics and denote by $\cR_n$ the sectional curvature of $g_n$. Then for any compact set $L\Subset M$ and any $\delta > 0$ there exists $n_0 = n_0(\delta)$ such that for all $n \geq n_0$, we have that on $L$ 
\begin{align}
\cR_n &\geq (k-\delta)g_n, \\
\text{resp.\ }\cR_n &\leq (k+\delta) g_n.\phantom{{resp.}}
\end{align}
\end{Proposition}
\begin{proof}
    Similar approximate inequalities are well-known for the Ricci curvature, see for example \cite[Lem.\ 4.1]{graf2020singularity}, \cite[Prop.\ 4.2]{kunzinger2022synthetic}, or in the case of vanishing Riemann curvature in the context of flatness, see \cite[Prop.\ 3.18]{KOV24}, and we shall base our arguments on these sources.
    For definiteness, say that the bound is below by $k$. 
    As a first step, fix $X,Y \in C^\infty(TM)$. Then by \cite[Prop.\ 4.2(ii) and Rem.\ 4.4]{kunzinger2022synthetic},
    given any compact $L\comp M$ and any $\delta>0$ there exists $N\in \N$ such that for $n\ge N$, 
    \begin{equation}\label{eq:sec-XY-est}
       |\cR_n(X\wedge Y,X \wedge Y) - \cR(X\wedge Y,X \wedge Y)\star_M \rho_{\eps_n^n}|<\delta  
    \end{equation}
    on $L$.
    Next, note that without loss of generality, we may assume $L$ to be small enough to admit a local frame $X_1,\dots, X_d$ ($d=\dim M$) of smooth vector fields on a neighborhood  $U$ of $L$ which is orthonormal with respect to some smooth background metric $h$.  By local equivalence of Riemannian metrics, there exist constants $c, C >0$ such that $cg(v,v)\le h(v,v)\le Cg(v,v)$ for each $v\in TM|_L$. Given
    $\delta>0$, due to \eqref{eq:sec-XY-est} there exists some $n_0\in \N$ such that, for $n\ge n_0$, 
    \begin{equation*}
        \left| \sum_{i,j,k,l=1}^d \lambda_i \lambda_j \mu_k \mu_l \big(\cR_n(X_i\wedge X_k,X_j\wedge X_l) - 
        \cR(X_i\wedge X_k,X_j\wedge X_l)\star_M \rho_{\eps_n^n})\big)\right| <\frac{\delta}{3C^2}
    \end{equation*}
    on $L$ for all $(\lambda_1,\dots,\lambda_d)$ and $(\mu_1,\dots,\mu_d)\in \R^d$ with $\sum_i \lambda_i^2 = \sum_i \mu_i^2 =1$. 
    Interchanging the sums with $\star_M$ here (which is permissible since the coefficients are constant) and setting 
    $V:=\sum_i \lambda_i X_i$, $W:= \sum_j \mu_jX_j$, $V$, $W$ are $h$-unit vector fields and 
    \begin{equation}\label{eq:R-estVW}
        \left| \cR_n(V\wedge W, V\wedge W) - \cR(V\wedge W,V\wedge W)\star_M \rho_{\eps_n^n} \right| <\frac{\delta}{3C^2}
    \end{equation}
    on $L$, for any choice of $\lambda_i, \mu_j$ as above. Similarly, using 
    \cite[Prop.\ 4.2(iii)]{kunzinger2022synthetic} it follows that we may additionally suppose $n_0$ to be chosen so large that, for all $n\ge n_0$,
    \begin{equation}\label{eq:g-estVW}
        |k||g(V\wedge W,V\wedge W)\star_M\rho_{\eps_n^n} - g_n(V\wedge W,V\wedge W)| < \frac{\delta}{3C^2}
    \end{equation}
    on $L$. By the assumed distributional sectional curvature bound, $\cR(V\wedge W,V\wedge W)-kg(V\wedge W,V\wedge W)\ge 0$, so non-negativity of $\rho$ gives 
    \[
    \cR(V\wedge W,V\wedge W)\star_M \rho_{\eps_n^n}-kg(V\wedge W,V\wedge W)\star_M \rho_{\eps_n^n}\ge 0.
    \]
    Combining this with \eqref{eq:R-estVW} and \eqref{eq:g-estVW}, we obtain
    \begin{equation*}
     \cR_n(V\wedge W,V\wedge W) - k g_n(V\wedge W,V\wedge W) > - \frac{2\delta}{3C^2}.
    \end{equation*}
    This inequality then also holds for $V, W$ replaced by individual $h$-unit tangent vectors $v, w\in TM|_L$. If, 
    in addition, $v$ and $w$ are $g$-orthogonal, then
    \begin{equation*}
    \begin{split}
     \cR_n(v\wedge w,v\wedge w) - k g_n(v\wedge w,v\wedge w) > - \frac{2\delta}{3C^2}h(v,v)h(w,w) \ge 
     - \frac{2\delta}{3} g(v\wedge w,v\wedge w).
    \end{split}
    \end{equation*}
    Increasing $n_0$ one last time and taking into account the uniform convergence of $g_n$ to $g$, we obtain
    that for each $n\ge n_0$ and all $g$-orthogonal $h$-unit vectors $v, w\in TM|_L$,
    \[
    \cR_n(v\wedge w,v\wedge w)  \ge (k-\delta) g_n(v\wedge w,v\wedge w).
    \]
    Finally, observing that each $g_n$-nondegenerate $2$-plane in any $T_xM$ is spanned by $g$-orthogonal $h$-unit vectors 
    concludes the proof. \qedhere
    
\end{proof}

Note that since one can extend tangent vectors to compactly supported vector fields, the result above provides a varying local sectional curvature bound for the smooth approximating metrics $g_n$, for which classical sectional curvature theory can then be used.

For locally Lipschitz metrics, a distributional bound on the curvature tensor of $g$ still implies local approximate bounds for the metrics $g_n$, but without control on the constants.

\begin{Proposition}[Local sectional curvature bounds for approximating metrics - $C^{0,1}_\loc$ case] \label{Prop: Localseccurvapprox2}
Let $g$ be a $C^{0,1}_\loc$-Riemannian metric on $M$ with distributional sectional curvature bounded below (resp.\ above) by some constant. Let $g_n$ be the corresponding approximating metrics. Then for any compact set $L\Subset M$ there exists $K= K(L)\in  \R$ and $n_0\in \N$ such that for all $n\ge n_0$ we have that on $L$
\begin{align}
     \cR_n &\geq Kg_n, \\
\text{resp.\ } \cR_n &\leq K g_n. \phantom{{resp.}}
\end{align}
\end{Proposition}
\begin{proof}
    This follows by the same arguments as in the proof of Proposition \ref{Prop: Localseccurvapprox}, only using the more general Friedrichs-type result given in \cite[Prop.\ 3.1(ii)]{CGHKS2025hawking}.
\end{proof}

\subsection{Synthetic sectional curvature bounds}\label{subsection: syntheticalexandrov}

The geometric insight underlying the study of sectional curvature conditions on general metric (length) spaces is the Toponogov comparison theorem (sometimes also referred to as the Cartan--Alexandrov--Toponogov comparison theorem), which states that for a (smooth and connected) Riemannian manifold $(M,g)$, a bound on the sectional curvature can equivalently be captured by way of triangle comparison in its naturally induced length space $(M, d_g)$ (cf.\ \cite[Thm. 6.5.6]{burago2001course} or \cite[Thm.\ 1.1]{alexanderbishop2008triangle}). Building on this, one can develop a theory of synthetic curvature bounds for metric spaces. Such spaces are called Alexandrov spaces, and we refer to standard textbooks in the field (e.g.\ \cite{burago2001course, AKPAlexandrov, BridsonHaefliger}) for treatments and references on the subject. In this section, we summarize the basic notions and results that will be needed in the sequel.

As usual, our aim is to compare the geometry of metric spaces with that of the two-dimensional $k$-\emph{model plane} $\mathbb{M}^2(k)$ for varying parameters $k \in \R$. For $k > 0$, this will involve a comparison with the sphere of radius $\sfrac{1}{\sqrt{k}}$, for $k=0$, with the Euclidean plane and for $k < 0$, we compare with the hyperbolic plane of constant curvature $k$. To unify notation, we will use the symbol $\varpi^k$ to denote the diameter of $\mathbb{M}^2(k)$, i.e.\ $\varpi^k=\sfrac{\pi}{\sqrt{k}}$, if $k > 0$ and $\varpi^k= \infty$, otherwise.

More concretely, given three points $x,y,z \in X$, we first define the \emph{model triangle} $\Tilde{\Delta}^{k}(xyz) =: [\Tilde{x}\Tilde{y}\Tilde{z}]$ as any triangle in $\mathbb{M}^2(k)$ with
\begin{equation*}
    \lvert \Tilde{x} - \Tilde{y} \rvert = d(x,y), \quad \lvert \Tilde{y} - \Tilde{z} \rvert = d(y,z), \quad \lvert \Tilde{z} - \Tilde{x} \rvert = d(z,x).
\end{equation*}
If $\mathrm{Per}(x,y,z):=d(x,y)+d(y,z) + d(z,x) < 2 \varpi^k$, such a triangle can be chosen uniquely up to the action of an isometry. We then further introduce the \emph{model angle} $\cang{k}(x;y,z)$ at $x$, using the appropriate law of cosines in $\mathbb{M}^2(k)$, as the angle of the triangle $[\Tilde{x}\Tilde{y}\Tilde{z}]$ at $\Tilde{x}$.

\begin{Remark}[On model angles] \label{RemarkAngleContinuity}
    By the law of cosines for the respective $k$-model space (cf.\ \cite[Sec.\ 1.1]{AKPAlexandrov}), one easily deduces that, for fixed $x,y,z \in X$, the function $\R \ni k \mapsto \cang{k}(x;y,z)$ is continuous and non-decreasing. Moreover, for fixed $x,y,z \in X$ and $k \in \R$, the map $d \mapsto \cang{k}_{d}(x;y,z)$ is continuous with respect to uniform convergence of metrics and jointly continuous in the pair $(k,d)$.
\end{Remark}

\begin{Definition}[Curvature bounded below] (\cite[8.1, 8.2]{AKPAlexandrov})
    Let $(X,d)$ be a metric space. Given some $k \in \R$, we say that a quadruple of points $(p,x_1,x_2,x_3)$ in $X$ satisfies $\CBB(k)$-\emph{comparison} provided
    \begin{equation*}
        \mathrm{Per}(p,x_i,x_j)< 2\varpi^k \quad \text{for all } 1 \leq i,j \leq 3, i \neq j,
    \end{equation*}
    in which case $(p,x_1,x_2,x_3)$ is said to be \emph{admissible}, and if it fulfills the angle comparison
\begin{equation} \label{CBBAngleCondition}
    \cang{k}(p;x_1,x_2) + \cang{k}(p;x_2,x_3) + \cang{k}(p;x_3,x_1) \leq 2 \pi.
\end{equation}
In case every point in $X$ possesses an open neighborhood $U$ in which all admissible quadruples of points satisfy the $\CBB(k)$-comparison condition, we call the space $X$ \emph{locally} $\CBB(k)$ and the set $U$ a $k$-\emph{comparison region}.

More generally, if each point $p \in M$ admits a $k_p$-comparison neighborhood \emph{for some} $k_p \in \R$ (possibly depending on $p$), then we call $X$ \emph{variably} $\CBB$.
\end{Definition}

In our setting, we will mostly be concerned with metric spaces $(X,d)$ whose metric $d$ is \emph{intrinsic}. This means that
    \begin{equation*}
        d(x,y) = \inf_{\gamma} L(\gamma),
    \end{equation*}
where the infimum is taken over all continuous curves $\gamma$ connecting $x$ to 
$y$ and $L(\gamma)$ denotes their variational length. Such spaces are called \emph{length spaces}, where $d$ is referred to as a \emph{length metric}. Equivalently, we can define the \emph{induced intrinsic} metric $\hat{d}(x,y) = \inf_{\gamma} L(\gamma)$ associated with an arbitrary metric $d$ as above and impose the condition $d = \hat{d}$.

\begin{Remark}[Restriction vs.\ induced length metric] \label{RemarkLocalization}
    In the sequel, we will sometimes restrict a given length metric $d$ to an open subset $U \subseteq X$. It is then clear that, in general, this restriction will not coincide with the induced intrinsic metric on $U$. We claim, however, that \emph{locally} they do agree. For this, let $p \in U$ and define $r:=\frac{1}{3}d(p, \partial U) > 0$. Now, let $\gamma$ be a path connecting two arbitrary points $x,y \in B_r(p)$ with length $L(\gamma) < d(x,y) + r$. It follows that $\gamma$ cannot leave $U$, as we would otherwise be able to find some point $z \in X \setminus U$ on its image for which
    \begin{equation*}
        4r \leq d(x,z) + d(z,y) \leq L(\gamma) < d(x,y) + r \leq 3r,
    \end{equation*}
    a contradiction.
\end{Remark}

\begin{Remark}[Locality of $\CBB$] \label{RemarkLocality}
    Suppose any point $p$ in a length space $(X, d)$ possesses a neighborhood $U$ satisfying local $\CBB(k)$-comparison w.r.t.\ the induced intrinsic metric on $U$. Then $X$ itself is locally $\CBB(k)$. Indeed, by Remark \ref{RemarkLocalization}, we only need to guarantee that the size of a $k$-comparison ball $B_r(p)$ around $p$ satisfies $r \leq \frac{1}{3}d(p, \partial U)$.
\end{Remark}

\begin{Definition}[Curvature bounded above]\label{def:CAT}
(\cite[9.1]{AKPAlexandrov})
    Let $(X,d)$ be a metric space. We say that a quadruple $(p_1,p_2,x_1,x_2)$ of points in $X$ satisfies $\CAT(k)$-\emph{comparison} for some $k \in \R$ if one of the following conditions is satisfied:
    \begin{enumerate}
        \item $\hfill \cang{k}(p_1;x_1,x_2)\leq  \cang{k}(p_1;p_2,x_1) + \cang{k}(p_1;p_2,x_2).\hfill$
        \item $\hfill \cang{k}(p_2;x_1,x_2)\leq  \cang{k}(p_2;p_1,x_1) + \cang{k}(p_2;p_1,x_2).\hfill$
        \item One of the six model angles above is not well-defined.
    \end{enumerate}
    In particular we will call \emph{admissible} a quadruple for which all the model angles are well-defined.
    
    We call an open set $U\subseteq X$ a $\CAT(k)$-\emph{comparison region}, for $k\in \R$, if all quadruples of points in $U$ satisfy the $\CAT(k)$-comparison condition. We call the space $(X,d)$ \emph{locally} $\CAT(k)$ if every point in $X$ is contained in a $\CAT(k)$-{comparison region}. If $X$ itself is a $\CAT(k)$-comparison region, we call $(X,d)$ a $\CAT(k)$ space. 
    
    More generally, if every point $p\in X$ possesses an open neighborhood which is a $\CAT(k_p)$-comparison region for some $k_p\in \R$ (possibly depending on $p$), we call $(X,d)$ \emph{variably} $\CAT$.
\end{Definition}

\subsection{Some results on incomplete Riemannian manifolds} \label{subsec: incompleteRiemann}

We now generalize some classical results from (smooth) Riemannian geometry, generally stated for complete Riemannian manifolds, to appropriately small domains in incomplete Riemannian manifolds. This conforms to the intuition that local results in Riemannian geometry hold on incomplete Riemannian manifolds as long as the necessary attention is given to the size of the neighborhoods in question. Our motivation to do this is that we will often be interested in the behavior of our smooth approximating metrics $g_n$ on small domains, which will not be complete when regarded as Riemannian manifolds. Additionally, we will need these generalizations to prove our main theorems for potentially incomplete $(M,g)$.

In this section, $(N,h)$ will always denote an arbitrary (connected) smooth Riemannian manifold.  

We define the \emph{completeness radius} at $p\in N$ as
\begin{equation*}
        \CR_h(p):=\sup\{r\in \R^+ : \overline{B_r(p)}\text{ is compact}\}\in \R^+\cup \{+\infty\}
    \end{equation*}
and the \emph{simple connectedness radius} at $p\in N$ as
    \begin{align*}
        \SR_h(p):=\sup\left\{R\in\R^+: {B_r(p)}\text{ is simply connected }\forall r\in (0,R]\right\}\in \R^+\cup \{+\infty\}.
    \end{align*}
    As a consequence of the fact that the union of an ascending sequence of simply connected open sets is itself simply connected, the supremum in the definition of $\SR_h(p)$ is actually a maximum whenever it is finite.
    
    Finally, we recall the definition of the \emph{injectivity radius} at $p\in N$:
    \begin{align*}
        \inj_h(p):= \sup\{&r\in \R^+: \, \exp_p \colon B_r(0) \to N \text{ is well-defined}\\ 
            &\;\;\text{and it is a diffeomorphism onto its image}\} \in\R^+\cup \{+\infty\}.
    \end{align*}
\begin{Remark}[On the exponential map]\label{rk:expMaps}
    Observe that three possible things can get in the way of $\exp_p \colon B_r(0) \to N$ being a (well-defined) diffeomorphism onto its image:
    \begin{enumerate}
        \item $\exp_p(v)$ is not defined for some $v\in B_r(0)$.
        \item $\exp_p$ admits a critical point $v\in B_r(0)$.
        \item $\exp_p$ is not injective on $B_r(0)$, i.e.\ there exist two distinct geodesics $t\mapsto \exp_p(tv), s\mapsto \exp_p(sw)$ connecting $p$ to $\exp_p(v)=\exp_p(w)$.
    \end{enumerate}
    We prepare the following two remarks on these issues:
    \begin{itemize}
        \item As long as $r \leq \CR_h(p)$, $\exp_p$ is well-defined on $B_r(0)$ due to the extendability of the geodesic ODE solutions on the compact sets $\overline{B_{r-\eps}(p)}$ for every small $\eps >0$. \item  If we further assume that $(N,h)$ has $\sec\leq K$ for some $K\in \R$ then, by Conjugate Point Comparison \cite[Thm.\ 11.12]{LeeRiem}, $\exp_p$ does not admit critical points on $B_{r}(0)$ for $r\leq\varpi^K$.
    \end{itemize}
    
\end{Remark}
We will now adapt a result due to Klingenberg, which, in its original form, states that no homotopy class of paths connecting a point $p$ in a complete Riemannian manifold to some $q \in B_{r}(p)$ while remaining in this ball can contain more than one geodesic provided the sectional curvature is bounded above by $K \in \R$ and $r\leq \varpi^K$.

\begin{Lemma}[Klingenberg's lemma] \label{Lemma: Klingenberg}
    Let $(N,h)$ be a smooth connected Riemannian manifold satisfying $\sec \leq K$ for some $K\in \R$ and fix $p\in N$. Let $r\in \R^+$ with $r\leq {\varpi^K}\wedge \CR_h(p)$ and suppose $\gamma_0, \gamma_1$ are two distinct geodesics joining $p$ to $q\in B_r(p)$. If $\gamma_0$ and $\gamma_1$ are path homotopic in $N$, meaning that there exists a continuous map $G\colon [0,1]\times [0,1]\to N$ such that
    \begin{align*}
        G(\alpha,0)=p,\quad G(\alpha,1)=q, \quad G(0,t)=\gamma_0(t), \quad G(1,t)=\gamma_1(t),
    \end{align*}
    then there must exist $(\alpha_*,t_*)\in [0,1]^2$ such that $d(p,G(\alpha_*,t_*)))\geq r$.
\end{Lemma}
\begin{proof}
    We consider the lifted paths $\tilde\gamma_0$ and $\tilde\gamma_1$ on $T_pN$ defined by
    \begin{equation*}
        \tilde\gamma_i(t):=t\gamma'_i(0), \quad \text{for } i=0,1.
    \end{equation*}
    By construction, $\exp_p \circ \, \tilde\gamma_i = \gamma_i$ but $\tilde\gamma_0(1)\neq\tilde\gamma_1(1)$. Recalling Remark \ref{rk:expMaps}, we observe that the exponential map admits no critical points in $B_r(0)$ and so is a local diffeomorphism there. In particular, following the geodesic $\gamma_0$, we can find an $\eps\in (0,1]$ such that the family of nearby intermediate curves $\restr{G}{[0,\eps]\times [0,1]}$ can be lifted to a path homotopy $\tilde{G}\colon[0,\eps]\times[0,1]\to T_p N$ with
    \begin{align} \label{eq:KlingenbergPathHomotopy}
        \restr{G}{[0,\eps]\times [0,1]} = \exp_p\circ\, \tilde{G},\quad \tilde{G}(0,t)= \tilde\gamma_0(t),\quad \tilde{G}(\alpha,0)=0,\quad \tilde{G}(\alpha,1)=\tilde\gamma_0(1).
    \end{align}
    As this process can be repeated along the curve $\tilde{G}(\eps, \cdot)$, the set
    \begin{align*}
        I:=\left\{ \eps\geq 0 : \restr{G}{[0,\eps]\times [0,1]} \text{ can be lifted to a path homotopy } \tilde{G}\text{ on }T_pN \text{ satisfying } \eqref{eq:KlingenbergPathHomotopy}\right\}
    \end{align*}
    is open in $[0,1]$. Now assume, by way of contradiction, that $G(\alpha,t)\in B_r(p)$ for all $(\alpha, t)\in [0,1]\times[0,1]$. Then, by compactness and continuity, there must exist some $\delta>0$ with $G([0,1]\times[0,1])\subseteq \overline{B_{r-\delta}(p)}$. Hence, $I$ is also closed as we can extend any $\tilde{G}$ defined on $[0,\eps')\times[0,1]$ to $[0,\eps']\times[0,1]$. By connectedness, we have that $I=[0,1]$, leading to a contradiction as $\tilde G(1,\, \cdot \,)$ would then be a lift of $\gamma_1$ starting in $0$ but ending in $\tilde\gamma_0(1)\neq\tilde\gamma_1(1)$.
\end{proof}

\begin{Proposition}\label{thm:injEstCAT}
    Let $(N,h)$ be a smooth connected Riemannian manifold satisfying $\sec_h \leq K$ for some $K\in \R$. Fix $p\in N$. Assume that $p$ admits a simply connected open neighborhood $U$ such that $U\subseteq B_R(p)$ for some $R\leq \varpi^K\wedge \CR_h(p)$. Then $\inj_h(p)\geq d(p,\partial U)$ and $\SR(p)\geq d(p,\partial U)$.
\end{Proposition}
\begin{proof}
    Assume, by way of contradiction, that $\inj_h(p)<r:=d(p,\partial U)$. Using conjugate point comparison \cite[Thm.\ 11.12]{LeeRiem}, we infer that $\exp_p$ is well defined and admits no critical points in $B_R(0)\supseteq B_r(0)$. Hence, by Remark \ref{rk:expMaps}, there must exist two distinct geodesics $\gamma_0,\gamma_1$ in $B_r(p)$ connecting $p$ to some point $q\in B_r(p)$. Now, as $U\supseteq B_r(p)$ is simply connected, $\gamma_0$ is path homotopic to $\gamma_1$ within $U$. This however directly contradicts Lemma \ref{Lemma: Klingenberg} which implies that any path homotopy $G$ from $\gamma_0$ to $\gamma_1$ admits a point $G(\alpha_*,t_*)$ that lies outside of $B_R(p) \supseteq U$.
    
    Finally, the second inequality follows directly from the observation that $B_a(p)$ is contractible as soon as $a\leq \inj_h(p)$.
\end{proof}

\begin{Lemma}\label{lem:injEstCAT2}
    Let $(N,h)$ be a smooth connected Riemannian manifold satisfying $\sec_h \leq K$ for some $K\in \R$. Fix $p,q\in N$ and suppose $r\in \R^+$ satisfies
    \begin{align*}
        d(p,q) + r &\leq \varpi^K \wedge \CR_h(p)\wedge \SR_h(p)\\
        2d(p,q) + r &\leq \varpi^K \wedge \CR_h(q).
    \end{align*}
    Then $\SR_h(q) \geq r$, in particular $B_r(q)$ is simply connected, and $\inj_h(q)\geq r$.
\end{Lemma}

\begin{proof}
    For the sake of contradiction, assume that $\inj_h(q)< r$. As in the previous proof, it follows that $\exp_q$ is well defined and admits no critical points in $B_r(0)$ and so there must exist two distinct geodesics $\gamma_0,\gamma_1$ in $B_r(q)$ connecting $q$ to some point in $B_r(q)$. As $B_{r+d(p,q)}(p)$ is simply connected, $\gamma_0$ and $\gamma_1$ are path homotopic in $B_{r+d(p,q)}(p)$. Hence, Lemma \ref{Lemma: Klingenberg} implies the existence of a point $G(\alpha_*,t_*)$ such that $d(q,G(\alpha_*,t_*))\geq r+2d(p,q)$ on any path homotopy $G$ from $\gamma_0$ to $\gamma_1$. This, however, is not possible as the triangle inequality would then imply $G(\alpha_*,t_*)\notin B_{r+d(p,q)}(p)$. We can therefore conclude that $\inj_h(q)\geq r$ and so, as in Proposition \ref{thm:injEstCAT}, $\SR_h(q) \geq r$.
    \end{proof}

\begin{Proposition}\label{thm:convEstCAT}
    Let $(N,h)$ be a smooth connected Riemannian manifold satisfying $\sec_h \leq K$ for some $K\in \R$. Fix $p\in N$ and let $r\in\R^+ ,\, r\leq  \frac{1}{5}\left(\varpi^K \wedge \CR_h(p)\wedge \SR_h(p) \right)$. Then the following hold:
    \begin{enumerate}
    \item Between any two points $x,y\in B_r(p)$ there exists a unique minimizing geodesic.
    \item The minimizing geodesics in $B_r(p)$ vary continuously with their endpoints.
    \item $B_r(p)$ is convex.
    \end{enumerate}
\end{Proposition}
\begin{proof}
    We first observe that, by definition of the completeness radius, for all $q\in B_r(p)$, we have $\CR_h(q)\geq \frac{4}{5}\CR_h(p)$ and so    \begin{align*}
        \CR_h(q)\wedge \varpi^K \geq \frac{4}{5}\left(\varpi^K \wedge\CR_h(p)\wedge  \SR_h(p)\right)\geq 4r \geq 2r + 2d(p,q).
    \end{align*}
    Thus, we can apply Lemma \ref{lem:injEstCAT2} to get $\inj_h(q) \geq 2r$. In particular, for every pair $x,y\in B_r(p)$ we have $d(x,y) < \inj_h(x)$, which proves points \emph{(i)} and \emph{(ii)}. Now, for $x\in B_r(p)$, define
    \begin{align*}
        C_x:= \left\{y\in B_r(p): \text{the minimizing geodesic $\gamma$ from $x$ to $y$ satisfies } \Img\gamma\subseteq B_r(p) \right\}.
    \end{align*}
    By item \emph{(ii)}, the set $C_x$ is open in $B_r(p)$, and it is non-empty since $x\in C_x$. To prove that it is also closed, let $\{y_n\}_{n\in\N} \subseteq C_x$ such that $y_n\to y\in B_r(p)$ and denote by $\gamma$ the minimizing geodesic connecting $x$ to $y$. Again by $(ii)$, we know that $\Img \gamma \subseteq \overline{B_r(p)}$ and we can assume w.l.o.g.\ that $\gamma$ is parametrized on the interval $[0,1]$. Applying Hessian Comparison \cite[Thm.\ 11.7]{LeeRiem} (noting that, in their notation, $s'_{K} (r)$ is non-negative for $r< \frac{1}{2}\varpi^{K}$ and in our case $d(p,\gamma(t))<\frac{1}{2}\varpi^{K}$ for all $t \in [0,1]$), we get that
    \begin{align*}
        \varphi''(t)&= \Hess_{\gamma(t)}(d(p,\,\cdot\,))[\gamma'(t), \gamma'(t)]\geq 0, \text{ where} \\
        \varphi(t)&:= d(p,\gamma(t)).
    \end{align*}
    Hence, $\varphi$ is convex, and so
    \begin{align*}
        \max_{t\in [0,1]} \varphi(t) = \max\{\varphi(0), \varphi(1)\} = \max\{d(p,x), d(p,y))\}<r,
    \end{align*}
    implying that $C_x$ is closed. Thus, we can conclude $C_x=B_r(p)$ for any $x \in B_r(p)$, which is precisely claim \emph{(iii)}.
\end{proof}

We define the \emph{total normality radius} at $p\in N$ as
\begin{align*}
    \tnorm_h(p)&:=\sup\{ R\in \R^+: B_r(p)\text{ satisfies \emph{(i), (ii), (iii)} from Theorem \ref{thm:convEstCAT} }\forall r\in (0,R]\},\\
    \tnorm_h(p)&\in \R^+ \cup \{+\infty\}.
\end{align*}
Observe that when $\tnorm_h(p)$ is finite, the supremum is actually a maximum. Let us conclude the section with the following compact rephrasing of the previous results.
\begin{Proposition}\label{thm:injconvCAT}
    For any smooth connected Riemannian manifold $(N,h)$ with $\sec_h\leq K$ and any $p\in N$ it holds that
    \begin{align*}
        \inj_h(p)&\geq \varpi^K \wedge \CR_h(p)\wedge \SR_h(p),\\ 
        \tnorm_h(p)&\geq \frac{1}{5}\left(\varpi^K \wedge \CR_h(p)\wedge \SR_h(p)\right).
    \end{align*}
\end{Proposition}

\section{Distributional bounds imply Alexandrov bounds}\label{Section: distrib implies alexandrov}

\subsection{Lower curvature bounds}\label{subsec: lower}
In this subsection, our main goal is to establish that the induced length space of a Riemannian metric of regularity $C^1$ satisfying a distributional lower sectional curvature bound is an Alexandrov space with the same sectional curvature bound in the sense of triangle comparison. More precisely, we aim to prove the following theorem.
\begin{Theorem}[Distributional lower bounds imply locally $\CBB$ in $C^1$] \label{ThmLowerBound}
    Let $(M,g)$ be a smooth connected manifold endowed with a $C^1$-Riemannian metric and denote by $d_g$ the Riemannian distance induced by $g$. Let $k \in \R$ and suppose $(M,g)$ satisfies a distributional sectional curvature bound from below by $k$. Then $(M,d_g)$ is locally $\CBB(k)$.
\end{Theorem}

\begin{Remark}\label{rk:localization}
    The above theorem can be generalized to the case of \textit{local bounds} on the distributional sectional curvature of $(M,g)$ in the sense of Definition \ref{def:distrSec}. In fact, in this case, for every relatively compact and open $U \subseteq M$, we would have    \begin{equation*}
        \mathcal{R} \geq k_U g \quad \text{in } \mathcal{D}'(U)
    \end{equation*}
    for some $k_U\in \R$. Applying the theorem as stated to $(U,g)$ gives us that $(U,d_U)$ is locally $\CBB(k_U)$. By Remark \ref{RemarkLocality}, we can then conclude that $(M,d_g)$ is \textit{variably} $\CBB$.
    
\end{Remark}

In the proof of this theorem, we will make essential use of the fact that comparison regions in locally $\CBB(k)$ spaces have a certain minimal relative size. To state this more explicitly, suppose $(X, d)$ is a length space which locally satisfies $\CBB(k)$-comparison, let $p \in X$ and consider the \emph{k-comparison radius}
\begin{equation*}
    c^k(p):= \sup \{ r > 0  :  B_r(p) \text{ is a \textit{k}-comparison region} \}.
\end{equation*}
It is easy to see that, for $p,q \in X$, we necessarily have $c^k(q) \leq c^k(p) + d(p,q)$. Indeed, otherwise the $k$-comparison region $B_{c^k(q)}(q)$ would properly contain $B_{c^k(p)}(p)$, contradicting maximality of $c^k(p)$. By symmetry, it then clearly also holds that $c^k(q) \geq c^k(p) - d(p,q)$, implying that $p \mapsto c^k(p)$ is $1$-Lipschitz \emph{continuous}. 

We now want to establish a local lower bound for $c^k(p)$, employing a result by Plaut \cite[Prop.\ 46]{PlautBook}, which reads as follows:

\begin{Proposition}
    Let $(X,d)$ be a complete length space and let $p \in X$. Suppose that $c^k \geq \kappa$ uniformly on a ball $B_r(p)$ of radius $r >0$ for some $\kappa > 0$. Then $c^k(p) \geq \frac{r}{3}$.
\end{Proposition}

\begin{Remark}
    In the above reference, Plaut's definition of bounded curvature invokes the notion of \emph{embedding curvature} attributed to Wald and Berestovskii. From this point of view, an open set $U$ within a metric space $X$ is considered to be a \emph{region of curvature $\geq k$} provided any quadruple of points in $U$ embeds isometrically into $\mathbb{M}^2(\ell)$ for some $\ell \geq k$, cf.\ \cite[Def.\ 22]{PlautBook}. To match this to our definition \eqref{CBBAngleCondition}, we invoke \cite[Prop.\ 23]{PlautBook}, which gives a direct translation of one into the other (note that our restriction on the perimeter of triangles is implicit in Plaut's formulation). Alternatively, we can follow the argument presented as a solution to Exercise 10.7 in \cite{AKPAlexandrov}, again keeping in mind our restriction on admissible quadruples.
\end{Remark}

\begin{Proposition} \label{PropMinimalSize}
    Suppose $U \subseteq X$ is open and relatively compact within a length space $(X,d)$ and is locally $\CBB(k)$ when considered with its own induced length metric. Then, for any $p \in U$, we have $c^k(p) \geq \frac{1}{9}d(p, \partial U)$.
\end{Proposition}

\begin{proof}
    Denote by $d_{\Bar{U}}$ the length metric induced on $\Bar{U}$. Compactness of $(\Bar{U}, \restr{d}{\Bar{U}})$ and an Arzelà–Ascoli argument ensure that the length space $(\Bar{U}, d_{\Bar{U}})$ is complete. The curvature assumption $c^k > 0$ on $U$ is stated with respect to its induced length metric $d_U$, which need not coincide with $d_{\Bar{U}}$. However, by Remark \ref{RemarkLocalization}, the restrictions of both metrics agree on $B_{\Tilde{R}}(p)$, where $\Tilde{R} = \frac{1}{3}d(p,\partial U)$. Thus, setting $R:= \Tilde{R} - 3 \eps$ for some small
    $\varepsilon >0$, continuity of $q \mapsto c^k(q)$ and compactness of $\overline{B_R(p)} \subseteq U$ guarantee the existence of a uniform lower bound $\kappa > 0$ on the comparison radii of points in $\overline{B_R(p)}$. Therefore, \cite[Prop.\ 46]{PlautBook} applies and grants $c^k(p) \geq \frac{R}{3} = \frac{1}{9}d(p, \partial U) - \varepsilon$, which gives the claim by letting $\varepsilon \rightarrow 0$.
\end{proof}

\newpage
\begin{proof}[Proof of Theorem \ref{ThmLowerBound}]
We argue locally and by approximation, starting with the sequence of smooth Riemannian metrics $g_n$ constructed in Section \ref{section:Regularization} converging to $g$ in $C^1_\mathrm{loc}$ as $n \rightarrow \infty$. On any open and relatively compact set $U$, we can now bound the sectional curvatures of the metrics $g_n$ from below by $k - \delta_n$, where $\delta_n \to 0$ as $n \to \infty$. This is a consequence of Proposition \ref{Prop: Localseccurvapprox}.

Hence, by Toponogov's theorem, we may conclude that the length space $(U, d_U^n)$ with local induced intrinsic metric $d_U^n$ derived from $g_n$ satisfies local $\CBB(k_n)$-comparison, where $k_n := k -\delta_n$. That is, given any $p \in U$, we can find a $d_U^n$-comparison neighborhood $\Omega_n$ in which any admissible quadruple of points $(q,x,y,z)$ satisfies 
\begin{equation*}
    \cang{k_n}_{d^n}(q;x,y) + \cang{k_n}_{d^n}(q;y,z) + \cang{k_n}_{d^n}(q;z,x) \leq 2 \pi.
\end{equation*}
Recalling Remark \ref{RemarkAngleContinuity}, we would now like to pass to the limit as $k_n \nearrow k$ and $d^n \rightarrow d$ uniformly on $U\times U$, inferring that $(U, d_U)$ satisfies local $\CBB(k)$-comparison. Here, however, is where it becomes crucial to establish a minimal size of the $\Omega_n$'s, as they could otherwise narrow down indefinitely, potentially prohibiting $d_U$-comparisons around $p$. To this end, let $\eps >0$ and observe that, by uniform convergence of the distances as in \eqref{eq:distance_sandwich2}, we know that for $n$ large enough,
\begin{equation*}
    \lvert d^n(p, q) - d(p, q) \rvert < \eps \quad \text{for all } q \in U.
\end{equation*}
Hence, setting $R:=\frac{1}{9}d(p,\partial U)$ and $R_n:=\frac{1}{9}(d^n(p, \partial U) - \eps)$, Proposition \ref{PropMinimalSize} directly implies that $B^{d^n}_{R_n}(p)$ is a $k_n$-comparison region for $d^n$ and by definition $R_n \geq R - \eps$.
Choosing $\eps >0$ sufficiently small, we have thus found a positive lower bound on the radii of $k_n$-comparison balls, which is \emph{uniform} in $n$ and therefore allows us to pass to the limit as anticipated. Finally, the locality of the curvature condition discussed in Remark \ref{RemarkLocality} guarantees that, in fact, the whole manifold $M$ locally satisfies $\CBB(k)$.
\end{proof}

We conclude this subsection with the following result for locally Lipschitz metrics.

\begin{Theorem}[Distributional lower bounds imply variable $\CBB$ in $C^{0,1}_{\mathrm{loc}}$]
    Let $(M,g)$ be a smooth connected manifold endowed with a $C^{0,1}_{\loc}$-Riemannian metric and denote by $d_g$ the Riemannian distance induced by $g$. If the sectional curvature of $(M,g)$ is locally distributionally bounded from below, then $(M,d_g)$ is variably $\CBB$.
\end{Theorem}

\begin{proof}
    We can assume w.l.o.g.\ that the distributional sectional curvature bound holds globally on $M$ as we could otherwise localize the argument and apply Remark \ref{RemarkLocalization}. Start by fixing a point $p \in M$ and choose some open and relatively compact neighborhood $U$ of $p$. Denoting by $(g_n)_{n \in \N}$ the sequence of approximating metrics for $g$, we have
    \begin{equation*}
        \cR_n \ge Kg_n
        \quad \textnormal{on } U
    \end{equation*}
    for some $K \in \R$ 
    by Proposition \ref{Prop: Localseccurvapprox2}. Hence, Toponogov's theorem ensures that $(U, d^n_U)$ satisfies local $\CBB(K)$-comparison. As $d_n \rightarrow d$ uniformly on $U \times U$, cf.\ \eqref{eq:distance_sandwich}, we can again follow the argument presented in the proof of Theorem \ref{ThmLowerBound} to infer that the size of any $K$-comparison ball around $p \in U$ is bounded from below by some radius $r > 0$ uniformly in $n \in \N$. Passing to the limit, we thus conclude that $B_r(p)$ with its induced intrinsic metric is $\CBB(K)$ and by Remark \ref{RemarkLocality}, so is $(B_r(p), \restr{d}{B_r(p)})$ upon potential rescaling.
\end{proof}

\subsection{Upper curvature bounds} \label{subsec: above}
We now focus our attention on the analogous results for upper curvature bounds.
\begin{Theorem}[Distributional upper bounds imply locally $\CAT$ in $C^1$]\label{thm:CATmain}
    Let $(M,g)$ be a smooth connected manifold endowed with a $C^1$-Riemannian metric and denote by $d_g$ the Riemannian distance induced by $g$. Let $k\in \R$ and suppose $(M,g)$ has distributional sectional curvature bounded from above by $k$. Then $(M,d_g)$ is locally $\CAT(k)$.
\end{Theorem}

\begin{Remark}
    In analogy to Remark \ref{rk:localization}, one can localize Theorem \ref{thm:CATmain} to get, under the hypothesis that $(M,g)$ has distributional sectional curvature locally bounded from above, that $(M,d_g)$ is variably $\CAT$.
\end{Remark}

Denote again by $g_n$ the smooth metrics approximating $g$ constructed in Section \ref{section:Regularization}. As in the case of curvature bounded below (Theorem \ref{ThmLowerBound}), the first step towards the proof of Theorem \ref{thm:CATmain} will be to estimate sizes of comparison regions uniformly in $n \in \N$.

{\begin{Remark}\label{rk:CATunifRad}
Proposition \ref{thm:injEstCAT} allows us to find, for any $p\in M$, a positive uniform bound from below to $\SR_{g_n}(p)$ and $\CR_{g_n}(p)$. More concretely, choose any relatively compact, open and connected neighborhood $U$ of $p$. Proposition \ref{Prop: Localseccurvapprox} then ensures that $\sec_n \leq k+1$ on $U$ for $n\in \N$ big enough. By the locally uniform convergence $d^n:=d_{g_n}\to d:=d_g$, we can assume, if $k+1>0$, that $U\subseteq B_{\varpi^{k+1}}^n(p)$ for all $n \in \N$ and we know that $d^n(p,\partial U) \geq R_1$ for some $R_1>0$. Thus, in particular $\CR_{g_n}(p)\geq R_1$ for all $n\in \N$. Furthermore, as $U$ is locally Euclidean, there must exist a simply connected open neighborhood $V$ of $p$ within $U$. Again for some $R_2\in (0,R_1]$, we then have $d^n(p,\partial V) \geq R_2$ and, in particular, due to Proposition \ref{thm:injEstCAT}, 
\begin{equation} \label{eq:CATunifRad}
    \SR_{g_n}(p)\geq R_2 \quad \text{and } \quad \CR_{g_n}(p) \geq R_2 \quad \text{for all }n\in \N
\end{equation}
Observe that the estimates on $\CR_{g_n}(p)$ and $\SR_{g_n}(p)$ hold for $(M,g_n)$ as well as for $(U,g_n)$.
\end{Remark}}

\begin{Proposition}\label{prop:CATball}
    For every $p\in M$ there exists a sequence $\delta_n\searrow 0$ and $R>0$ such that that $(\overline{B_R^n(p)}, d^n)$ is a $\CAT(k+\delta_n)$ space for every $n\in \N$.
\end{Proposition}

\begin{proof}
    Fix an open and relatively compact neighborhood $U$ of $p$. Invoking Proposition \ref{Prop: Localseccurvapprox}, we can define a sequence $\delta_n\searrow 0$ such that $\sec_n\leq k+\delta_n$ on $U$. In particular, the length spaces $(U,d^n_U)$ are locally $\CAT(k+\delta_n)$. Without loss of generality, we can moreover assume that $\delta_n\leq 1$ as well as, in the case $k+1>0$, $U\subseteq B_{\varpi^{k+1}}^n(p)$, for all $n \in \N$. Applying \eqref{eq:CATunifRad} to Proposition \ref{thm:injconvCAT}, we now get that $\tnorm_{g_n}(p)\geq \frac{1}{5}(\varpi^{k+1}\wedge R_2)=:R'$ for all $n\in \N$. Fix $R\in (0, R')$. Then
    \begin{align*}
        \overline{B_R^n(p)}= \bigcap_{r\in (R,R')} B^n_{r}(p)
    \end{align*}
    is convex and therefore inherits the length structure of $(U,g_n)$, in particular, each $\overline{B^n_r(p)}$ is locally $\CAT(k+\delta_n)$. Crucially, however, $\overline{B_R^n(p)}$ is complete. Recalling that we have already established uniqueness of minimizing geodesics and continuous variation with their endpoints, we can therefore apply the Patchwork Globalization Theorem \cite[9.30]{AKPAlexandrov} to infer that $(\overline{B_R^n(p)},d^n_U)$ is, in fact, (globally) $\CAT(k+\delta_n)$ for every $n \in \N$. We conclude by observing that the convexity of $\overline{B_R^n(p)}$ in $(M,g_n)$ ensures that $d_U^n= d^n$ on $\overline{B_R^n(p)}\times \overline{B_R^n(p)}$.
\end{proof}

\begin{proof}[Proof of Theorem \ref{thm:CATmain}]
    Let us fix $p\in M$ and take $R>0$ and a sequence $\delta_n\searrow 0$ as in Proposition \ref{prop:CATball}. Fix any admissible quadruple of points $( a,b,x,y )$ in $ B_{R/2}(p)$. Assume $a,b,x,y$ to be distinct, else the $\CAT$-comparison is trivial. Observe that for all $n\in \N$ big enough it holds that $B_{R/2}(p)\subseteq \overline{B^{n}_R(p)}$. This means that one of the following two conditions holds:
    \begin{align}
        \cang{k+\delta_n}_{d^{n}}(a;x,y)&\leq  \cang{k+\delta_n}_{d^{n}}(a;b,x) + \cang{k+\delta_n}_{d^{n}}(a;b,y),\label{eq:CATtemp1}\\
        \cang{k+\delta_n}_{d^{n}}(b;x,y)&\leq  \cang{k+\delta_n}_{d^{n}}(b;a,x) + \cang{k+\delta_n}_{d^{n}}(b;a,y).
    \end{align}
    We can then always restrict to a subsequence indexed by $\{n_j\}_{j\in \N}$ such that one of the two conditions holds for all $n_j$. Assume this condition is (\ref{eq:CATtemp1}), the other case is analogous. Now take
    \begin{align*}
        \cang{k+\delta_{n_j}}_{d^{n_j}}(a;x,y)&\leq  \cang{k+\delta_{n_j}}_{d^{n_j}}(a;b,x) + \cang{k+\delta_{n_j}}_{d^{n_j}}(a;b,y)
    \end{align*}
    to the limit $j\to \infty$. Due to the uniform convergence of the distance functions and the (joint) continuity of the model angle function (Remark \ref{RemarkAngleContinuity}) we get
    \begin{align*}
        \cang{k}_d(a;x,y)\leq  \cang{k}_d(a;b,x) + \cang{k}_d(a;b,y),
    \end{align*}
    which confirms the $\CAT(k)$-comparison.
\end{proof}
We conclude this subsection with a generalization of Theorem \ref{thm:CATmain} to locally Lipschitz metrics. The proofs will only be sketched as they are analogous to those in the previous part.
\begin{Proposition}\label{prop:CATball2}
    Let $M$ be a smooth connected manifold endowed with a $C^{0,1}_\loc$-Riemannian metric $g$. Assume that $(M,g)$ has distributional sectional curvature locally bounded from above. Then for every $p\in M$ there exists $K\in \R$ and $R>0$ such that $(\overline{B_R^n(p)}, d^n)$ is a $\CAT(K)$ space for every $n\in \N$.
\end{Proposition}
\begin{proof}
    Fix a relatively compact and open neighborhood $U$ of $p$. By assumption, the sectional curvature is bounded from above on $(U,g)$ in the distributional sense by some constant. Hence, by Proposition \ref{Prop: Localseccurvapprox2} there exists $K\in \R$ such that $\sec_n\leq K$ on $U$ for all $n\in \N$. In particular $(U,d^n_U)$ are locally $\CAT(K)$ spaces. Up to restricting $U$ we can assume, if $K>0$, that $U\subseteq B_{\varpi^{K}}^n(p)$ for all $n$. By applying the procedure from Remark \ref{rk:CATunifRad} we obtain uniform estimates to $\SR(p)$ and $\CR(p)$. By Theorem \ref{thm:injconvCAT} we get that $\tnorm_{g_n}(p)\geq \frac{1}{5}(\varpi^{K} \wedge R_2)=R'$ for all $n\in \N$. Fix $R\in (0, R')$. The rest of the proof is analogous to Proposition \ref{prop:CATball}.
\end{proof}
\begin{Theorem}[Distributional upper bounds imply variable $\CAT$ in $C^{0,1}_{\mathrm{loc}}$]
    Let $(M,g)$ be a smooth connected manifold endowed with a $C^{0,1}_{\loc}$-Riemannian metric and denote by $d_g$ the Riemannian distance induced by $g$. Let $k\in\R$ and suppose $(M,g)$ has distributional sectional curvature bounded from above by $k$. Then $(M,d_g)$ is variably $\CAT$.
\end{Theorem}
\begin{proof}
    Let us fix $p\in M$ and take $R>0, K\in \R$ as in Proposition \ref{prop:CATball2}. Fix a quadruple of pairwise distinct points $a,b,x,y$ in $B_{R/2}(p)$. Analogously to Theorem \ref{thm:CATmain}, it is not restrictive to assume that the following holds for a sequence $\{n_j\}_{j\in \N}$
    \begin{align*}
        \cang{K}_{d^{n_j}}(a;x,y)\leq\cang{K}_{d^{n_j}}(a;b,x) + \cang{K}_{d^{n_j}}(a;b,y).
    \end{align*}
    Taking $j\to \infty$ yields
    \begin{align*}
        \cang{K}_{d}(a;x,y)\leq\cang{K}_{d}(a;b,x) + \cang{K}_{d}(a;b,y).
    \end{align*}
    This confirms the $\CAT(K)$-comparison. Remember that here $K$ depends on $p\in M$. Thus, we only conclude that $(M,d)$ is a variably $\CAT$ space.
\end{proof}

\subsection{Examples}\label{Subsec: examples}

We conclude this section by analyzing two classical examples due to P.\ Hartman and A.\ Wintner \cite{HaWi1951geod} which show that, in the absence of distributional curvature bounds, CAT or CBB-properties may fail to hold.

In analogy to the smooth case, one may wonder whether 
the rigid properties of Alexandrov spaces (such as geodesic non-branching for $\CBB$- or geodesic uniqueness for $\CAT$-spaces) continue to hold for Riemannian metrics of lower regularity. However, it turns out that, in general, this is no longer the case for regularities weaker than $C^{1,1}_\loc$. An example for this in $C^{1,\alpha}_\loc$, with arbitrary $\alpha<1$, has been given in \cite{HaWi1951geod}. Consider, for $\lambda\in (1,2)$,
\begin{gather*}
    M=\R_x\times \R_y,\qquad
    g:=\begin{pmatrix}
            1+\vert x \vert^\lambda & 0\\
            0 & 1+\vert x \vert^\lambda
        \end{pmatrix}.
\end{gather*}

Here, one can compute the (distributional) Riemannian tensor, obtaining
\begin{align*}
     \dif x ( \Riem (\partial_x, \partial_y) \partial_y)= -\frac{\lambda(\lambda -1)\vert x \vert^{\lambda - 2}}{2(1+ \vert x\vert^\lambda)^2} 
     + \frac{\lambda \vert x \vert^{2\lambda - 2}}{2(1+ \vert x\vert^\lambda)^2}.
\end{align*}
Thus, the sectional, i.e., in this case, the Gauss curvature
\[
\sec(\partial_x,\partial_y) = \frac{\lambda |x|^{\lambda-2}(1+|x|^\lambda -\lambda)}{2(1+|x|^\lambda)^3}
\]

blows up to $-\infty$ as $x\to 0$. Consequently, $g$ does not possess a lower sectional curvature bound in the distributional sense. However, it satisfies a bound from above. Therefore, we know that $(M,d_g)$ is locally $\CAT$ by Theorem \ref{thm:CATmain}. This is in line with \cite[Sec.\ 7]{HaWi1951geod}, in which Hartman and Wintner prove that all points $(x,y)\in M$ are joined to $(0,0)$ by exactly one geodesic. Since distance realizers for $C^1$-Riemannian metrics necessarily satisfy the geodesic equation, any such geodesic must then be the unique distance minimizer between $(x,y)$ and $(0,0)$, which is precisely the behavior expected from a $\CAT$ space. Moreover, the authors prove the existence of multiple geodesics branching along the line $\{x=0\}$, and again by uniqueness of connecting geodesics all of these curves are minimizing. 
This implies that $(M,d_g)$ is not a metric space with sectional curvature bounded from below in the synthetic sense, as $\CBB$ implies non-branching of minimizing geodesics.

An example in $C^{1,\alpha}$ of the opposite phenomenon again comes from \cite[Sec.\ 5]{HaWi1951geod}. For this, fix $\lambda\in (1,2)$ and let
\begin{gather*}
    M=(-1,1)_x\times \R_y,\qquad
    g:=\begin{pmatrix}
            1 & 0\\
            0 & 1-\vert x \vert^\lambda
        \end{pmatrix}.
\end{gather*}
Then
\[
\dif x ( \Riem (\partial_x, \partial_y) \partial_y)= \frac{\lambda}{2}(\lambda-1)|x|^{\lambda-2} + \frac{\lambda^2 \vert x \vert^{2\lambda - 2}}{4- 4\vert x\vert^\lambda},
\]
\[
\sec(\partial_x,\partial_y) = \frac{\lambda |x|^{-2+\lambda}(2\lambda +|x|^\lambda(2-\lambda)-2)}{4(|x|^\lambda-1)^2}.
\]

Here, the sectional curvature blows up to $+\infty$ as $x\to 0$, while being bounded from below. Hence, Theorem \ref{ThmLowerBound} ensures that $(M,d_g)$ is locally $\CBB$ and so we know that minimizing geodesics do not branch. Indeed, this is also a direct consequence of the fact that, in this example, the initial value problem for geodesics admits a unique solution. On the other hand, all sufficiently close points on the line $\{x=0\}$ admit two minimizing geodesics (and infinitely many non-minimizing curves satisfying the geodesic ODE). This implies that $(M,d_g)$ is not a $\CAT$-space.

\section{Conclusion \& outlook}\label{section: Conclusion and outlook}

In this work, we have introduced sectional curvature bounds based on distributional inequalities for the curvature tensor of metrics in the Geroch--Traschen class. We have shown that if the metric is $C^1$, then these distributional curvature bounds imply the corresponding ones from the theory of Alexandrov spaces.

Naturally, the next question to investigate would be a converse of our results for $g \in C^1$:

\begin{Conjecture} \label{conjecture: C1converse}
    Let $(M,g)$ be a Riemannian manifold with $g \in C^1$ such that $(M,d_g)$ is an Alexandrov space with curvature bounded from below (resp.\ above) by $k \in \R$. Then the distributional sectional curvature of $(M,g)$ is bounded from below (resp.\ above) by $k$.
\end{Conjecture}

A result of similar nature, albeit with slightly different hypotheses, can be found in \cite{lebedevapetrunin2024curvature}.

Similar compatibility questions in the case of Ricci curvature bounds (below) have recently been answered (up to technical details) to the positive for continuous Geroch--Traschen metrics by Mondino--Ryborz \cite{mondino2024equivalence} (see also \cite{KOV24}). There, the converse direction (i.e., synthetic $\Rightarrow$ distributional) involved the use of numerous tools from the synthetic theory of $\mathrm{RCD}$-spaces. Essentially, it used that synthetic Ricci curvature bounds below can be expressed in terms of a bound on a measure-valued Ricci tensor introduced by Gigli (see \cite{GigliPasq}). We expect that, similarly, the proof of Conjecture \ref{conjecture: C1converse} will rely heavily on results from the analysis of Alexandrov spaces.

In this context, at least in the case of curvature bounded below, the problem may be related to the detailed study of the Riemann curvature tensor introduced by Gigli in general $\mathrm{RCD}$-spaces (see \cite{GigliRiemann}). To the authors' knowledge, no link between bounds of the Riemann tensor and the theory of Alexandrov spaces has yet been established. One should further keep in mind that sectional curvature bounds on smooth Riemannian manifolds have been characterized using methods from optimal transport (see the work of Ketterer--Mondino \cite{KettererMondinoSectional}), which are much more commonplace in the study of Ricci curvature bounds. The link between these methods and distributional bounds for low regularity metrics is certainly worth investigating. Ultimately, the goal would be to show the equivalence between distributional sectional curvature bounds and Alexandrov bounds for continuous Geroch--Traschen metrics, in analogy with Mondino-Ryborz \cite{mondino2024equivalence}. 

Moreover, analogous questions concerning the compatibility of distributional and synthetic notions of sectional curvature bounds arise naturally for metrics of indefinite signature \cite{alexanderbishop2008triangle}. In particular, timelike sectional curvature bounds in Lorentzian signature \cite{BKOR} are of fundamental importance in General Relativity and the geometry of spacetimes.

\section*{Acknowledgments}
This research was funded in part by the Austrian Science Fund (FWF) [Grants DOI \linebreak \href{https://doi.org/10.55776/PAT1996423}{10.55776/PAT1996423},   \href{https://doi.org/10.55776/EFP6}{10.55776/EFP6} and \href{https://doi.org/10.55776/J4913}{10.55776/J4913}]. For open access purposes, the authors have applied a CC BY public copyright license to any author-accepted manuscript version arising from this submission. Argam Ohanyan was partly supported by the Canada Research Chairs program CRC-2020-00289 and Natural Sciences and Engineering Research Council of Canada Grant RGPIN-2020-04162. Alessio Vardabasso acknowledges the support of the Vienna School of Mathematics.

\addcontentsline{toc}{section}{References}

\end{document}